\pdfoutput=1
\documentclass[a4paper,10pt]{amsart}
 \usepackage{
  amsmath,   amscd,    amssymb,   mathrsfs,
  mathtools, fullpage, enumerate, thmtools,
  fancyhdr,  stmaryrd, tikz-cd,   etoolbox,
  chngcntr,  yfonts}
\usepackage[shortlabels]{enumitem}
\usepackage[utf8]{inputenc}

\setlist{leftmargin=2\parindent}

\usepackage[pagebackref, hypertexnames=false, bookmarks=false]{hyperref}

\usepackage[capitalize]{cleveref}

\newenvironment{smatrix}{\left( \begin{smallmatrix} } {\end{smallmatrix} \right) }

\newcommand{\stbt}[4]{\begin{smatrix}#1 & #2 \\ #3 & #4\end{smatrix}}

\theoremstyle{plain}
\newtheorem{theorem}{Theorem}[section]

\newtheorem{lemma}[theorem]{Lemma}
\newtheorem{proposition}[theorem]{Proposition}
\newtheorem{corollary}[theorem]{Corollary}
\newtheorem{definition}[theorem]{Definition}
\newtheorem{notation}[theorem]{Notation}

\theoremstyle{remark}
\declaretheorem[name=Remark,sibling=theorem,qed={\lower-0.3ex\hbox{$\diamond$}}]{remark}
\declaretheorem[name=Note,sibling=theorem,qed={\lower-0.3ex\hbox{$\diamond$}}]{note}

\newcommand{\ah}{\mathrm{ah}}

\DeclareMathOperator{\GL}{GL}
\DeclareMathOperator{\SL}{SL}

\DeclareMathOperator{\GSp}{GSp}

\DeclareMathOperator{\Res}{Res}

\DeclareMathOperator{\imp}{imp}

\DeclareMathOperator{\Spec}{Spec}
\DeclareMathOperator{\Sym}{Sym}
\DeclareMathOperator{\alg}{alg}

\DeclareMathOperator{\Nm}{Nm}

\DeclareMathOperator{\BC}{BC}

\DeclareMathOperator{\vol}{vol}
\DeclareMathOperator{\As}{As}

\renewcommand{\AA}{\mathbf{A}}

\newcommand{\CC}{\mathbf{C}}

\newcommand{\QQbar}{\overline{\QQ}}
\newcommand{\QQ}{\mathbf{Q}}

\newcommand{\Qp}{\QQ_p}
\newcommand{\fN}{\mathfrak{N}}
\newcommand{\ZZ}{\mathbf{Z}}
\newcommand{\Zp}{\ZZ_p}

\newcommand{\cE}{\mathcal{E}}

\newcommand{\cF}{\mathcal{F}}

\newcommand{\cH}{\mathcal{H}}

\newcommand{\cL}{\mathcal{L}}
\newcommand{\cM}{\mathcal{M}}

\newcommand{\cO}{\mathcal{O}}

\newcommand{\cS}{\mathcal{S}}

\newcommand{\cW}{\mathcal{W}}

\newcommand{\ch}{\mathrm{ch}}

\newcommand{\into}{\hookrightarrow}

\newcommand{\sgn}{\mathrm{sgn}}

\newcommand{\uPi}{\underline{\Pi}}

\newcommand{\pp}{\mathfrak{p}}

\numberwithin{equation}{section}
\renewcommand{\le}{\leqslant}

\renewcommand{\ge}{\geqslant}

\author{David Loeffler}
\address[Loeffler]{UniDistance Suisse, Schinerstrasse 18, 3900 Brig, Switzerland}
\email{david.loeffler@unidistance.ch}
\urladdr{\href{http://orcid.org/0000-0001-9069-1877}{0000-0001-9069-1877}}

\author{Sarah Livia Zerbes}
\address[Zerbes]{Department of Mathematics, ETH Z\"urich, R\"amistrasse 101, 8092 Z\"urich, Switzerland}
\email{sarah.zerbes@math.ethz.ch}
\urladdr{\href{http://orcid.org/0000-0001-8650-9622}{0000-0001-8650-9622}}

\thanks{D.L. gratefully acknowledges the support of the European Research Council through the Horizon 2020 Excellent Science programme (Consolidator Grant ``ShimBSD: Shimura varieties and the BSD conjecture'', grant ID 101001051)}

\title{Poles of $p$-adic Asai $L$-functions and distinguished representations}

\begin{document}
 \renewcommand{\crefrangeconjunction}{--} 

 \begin{abstract}
  We give a criterion in terms of $p$-adic Asai $L$-functions for a cuspidal automorphic representation of $\GL_2$ over a real quadratic field $F$ to be a distinguished representation, providing a $p$-adic counterpart of a well-known theorem of Flicker for the complex Asai $L$-function.
 \end{abstract}

 \subjclass[2020]{11F41, 
 11F67 
 }

 \maketitle

\section{Introduction}

 \subsection{Background: complex case}

  Let $\Pi$ be a cuspidal automorphic representation of $\GL_2 / F$, for a real quadratic field $F$. Then the \emph{Asai}, or \emph{twisted tensor}, $L$-series $L_{\As}(\Pi, s)$ is defined by an Euler product which converges for $\operatorname{Re}(s) > 1$. It always has meromorphic continuation to $\CC$, but the continuation is not always an entire function: it is holomorphic away from $s = 1$, but it can have a simple pole at $s = 1$. Similar properties hold more generally for the twisted Asai $L$-series $L_{\As}(\Pi, \chi, s)$ for $\chi$ a Dirichlet character. The occurrence of these poles is governed by the following theorem (which is due to Flicker \cite{flicker88} for $\chi = 1$, and can be extended to general $\chi$ using results of Lapid--Rogawski \cite{lapidrogawski98}):

  \begin{theorem}\label{thm:whenpoles}
   The following are equivalent:
   \begin{enumerate}
    \item The function $L_{\As}(\Pi, \chi, s)$ has a pole at $s = 1$.
    \item $\Pi$ is \emph{$\chi$-distinguished}, i.e.,~we have $\chi^{-2} = (\omega_{\Pi})|_{\QQ}$ (where $\omega_{\Pi}$ is the central character of $\Pi$) and the integral
    \[ \int_{\AA^\times \GL_2(\QQ) \backslash \GL_2(\AA)} \phi(h) \chi(\det h) \, \mathrm{d}h \]
    is non-zero for some $\phi \in \Pi$. (Note that the condition on $\chi$ is required for the integrand to be well-defined.)
   \end{enumerate}
   Moreover, there exist characters $\chi$ such that $\Pi$ is $\chi$-distinguished if, and only if, $\Pi$ is a twist of a base-change from $\GL_2 / \QQ$.
  \end{theorem}


  Thus the presence or absence of poles of $L_{\As}(\Pi, \chi, s)$ serves to identify representations with a $\GL_2(\AA)$-invariant period, and the representations with such a period are precisely those in the image of a specific functorial lifting -- both statements which admit a wealth of interesting generalisations to other automorphic forms and their $L$-series.

  If $\Pi = \BC(\pi) \times \psi$ is a twist of a base-change, then we have $L_{\As}(\Pi, \chi, s) = L_{\As}(\BC(\pi), \chi\psi|_{\QQ}, s)$, so we can suppose $\psi = 1$ without loss of generality. Then we have the factorisation formula (\cref{cor:factorisation} below)
  \[ \tag{\dag} L_{\As}(\Pi, \chi, s) = L(\Sym^2 \pi, \chi, s) \cdot L(\chi \omega_{\pi} \eta_F, s), \]
  where $\eta_F$ is the quadratic character associated to $F / \QQ$. (Compare e.g.~\cite[Ch.~4, Cor.~1]{ghatethesis} for imaginary quadratic fields.) It follows that $\Pi$ is always $\chi$-distinguished for $\chi = \eta_F\omega_{\pi}^{-1}$; and this is the only such character unless $\pi$ is dihedral (in which case there may be further poles coming from the $\Sym^2$ term).

 \subsection{P-adic case}

  In the recent paper \cite{grossiloefflerzerbesLfunct}, Grossi and the present authors defined a $p$-adic counterpart of the Asai $L$-function, assuming $p = \pp_1 \pp_2$ is split in $F$ and $\Pi$ is generated by a holomorphic eigenform $\cF$ which is ordinary at $\pp_1$. However, this function is \emph{always} a $p$-adic measure, i.e.,~an analytic function on the weight space; it never has poles, irrespective of the choice of $\cF$.

  In this $p$-adic setting, we do have an analogue of the factorisation formula $(\dag)$ (although this is a significantly deeper theorem than the complex version, and will be established in forthcoming work of D.~Krekov); and the $p$-adic zeta function has a pole at $s = 1$, like its complex counterpart. However, the $p$-adic symmetric square $L$-function has a ``trivial zero'' at $s = 1$, forced by the shape of the interpolation factors at $p$, and this trivial zero cancels out the pole of the $p$-adic zeta function. Hence the presence or absence of poles of the $p$-adic $L$-function can no longer be used to characterize which automorphic representations are distinguished.

  In this short note, we show that this problem can be circumvented via an ``improved'' $p$-adic Asai $L$-function, interpolating the values $L(\As \cF_m, \chi, 1-m)$ as $\cF_m$ varies through a Hida family of Hilbert eigenforms of weight $(k + 2m, k)$ for varying $m$ and fixed $k$. This function satisfies an interpolation formula, relating it to the Asai $L$-function of the weight $(k + 2m, k)$ specialisation at $s = 1 - m$; and this interpolation formula involves fewer Euler factors than the $p$-adic $L$-function of \cite{grossiloefflerzerbesLfunct} -- in particular, the Euler factor that gives the trivial zero does not appear. We show that this improved $p$-adic $L$-function genuinely can have a pole at $m = 0$; and, moreover, this pole occurs if and only if the weight $k$ specialisation generates a $\chi$-distinguished representation, giving a $p$-adic counterpart of \cref{thm:whenpoles}.

 \subsection*{Acknowledgements} It is a pleasure to thank the two anonymous referees for their insightful remarks on the original manuscript.


\section{Complex $L$-series and period integrals}

 We fix a real quadratic field $F$ of discriminant $D$, and a numbering of its real embeddings as $(\sigma_1, \sigma_2)$.

 \subsection{Asai $L$-functions}

   Let $\Pi$ be a (unitary) cuspidal automorphic representation of $\GL_2 / F$, generated by a cuspidal Hilbert modular newform $\cF$ of level $\fN$ and weight $(k_1, k_2)$. Then we define the Asai $L$-function $L_{\As}(\Pi, s)$ as in \S 1.2 of \cite{grossiloefflerzerbesLfunct}, with Euler factors at the bad primes determined via the local Langlands correspondence; this extends the original definition due to Asai \cite{asai77} in the case $\fN = 1$. We define twisted Asai $L$-functions $L_{\As}(\Pi, \chi, s)$ similarly.

   We also define the \emph{imprimitive} Asai $L$-function $L_{\As}^{\imp}(\Pi, \chi, s)$ as Definition 5.1.2 of \emph{op.cit.}, given by
   \[
    L_{\As}^{\imp}(\Pi, s) = L_{N_\chi \Nm(\fN)} \Big(\chi^2 \omega_{\Pi} |_{\QQ}, 2s\Big) \cdot \sum_{\substack{n \ge 1 \\ (n, N_\chi) = 1}} a^\circ_{n \cO_F}(\cF) \chi(n) n^{-s},
   \]
   where $a^\circ_{\mathfrak{m}}(\cF)$ is the (unitarily-normalised) Hecke eigenvalue of $\cF$ at the ideal $\mathfrak{m}$, and $L_{N_\chi \Nm(\fN)}(-)$ denotes the Dirichlet $L$-series with the Euler factors at primes dividing $N_\chi \Nm(\fN)$ omitted. These are related by
   \[
    L^{\imp}_{\As}(\Pi, \chi, s) = L_{\As}(\Pi, \chi, s) \cdot \prod_{\ell \mid N_{\chi} \operatorname{Nm}(\fN)} C_\ell(\Pi, \chi, \ell^{-s})
   \]
   for some polynomials $C_\ell(\Pi, \chi, X) \in \QQbar[X]$; moreover, the zeroes of the factors $C_\ell(\Pi, \chi, \ell^{-s})$ have real parts $\le 0$ (\cite[Proposition 5.1.3]{leiloefflerzerbes18}). Thus $L^{\imp}_{\As}(\Pi, \chi, s)$, like $L_{\As}(\Pi, \chi, s)$, has a simple pole at $s = 1$ when $\Pi$ is $\chi$-distinguished and is an entire function otherwise.

  \subsubsection*{Base-change case}

   Given a cuspidal automorphic representation $\pi$ of $\GL_2 / \QQ$, there is an automorphic representation $\Pi = \BC(\pi)$ of $\GL_2 / F$, the base-change (or Doi--Nagunuma lift) of $\pi$, which is cuspidal unless $\pi$ is induced from a Gr\"ossencharacter of $F$. Note that the central character of $\BC(\pi)$ is given by $\omega_{\Pi} = \omega_\pi \circ \Nm_{F / \QQ}$; and if $\pi$ is generated by a holomorphic modular form of weight $k$, then $\Pi$ is generated by a holomorphic Hilbert modular form of parallel weight $(k, k)$.

   \begin{corollary}\label{cor:factorisation}
    If $\Pi = \operatorname{BC}(\pi)$, then there is a factorisation of complex $L$-functions
    \[ L_{\As}(\Pi, \chi, s) = L(\Sym^2 \pi, \chi, s) \cdot L(\chi \eta_F \omega_{\pi}, s), \]
    for any $\chi$, where $L(\Sym^2 \pi, \chi, s)$ is the twisted symmetric square $L$-function.
   \end{corollary}

   \begin{proof}
    It is easily seen that the local Euler factors at unramified primes agree on both sides. To check the compatibility at all places, one can make a purely automorphic argument using compatibility of local and global base-change maps; or when $\pi$ is holomorphic (the only case we need below) we can argue that both sides are the $L$-functions of compatible families of semisimple Galois representations, and the traces of these agree at all but finitely many places, so they agree everywhere by Chebotarev density.
   \end{proof}


 \subsection{Period integrals}

  Let $\Pi$ be generated by a holomorphic newform $\cF$ of parallel weight $(k, k)$. We consider the normalised partially-holomorphic eigenforms $\cF^{\ah, i}$ (for $i = 1, 2$) defined as in \cite[Lemma 5.2.1]{leiloefflerzerbes18}, and their restrictions $\iota^*(\cF^{\ah, i})$ to the upper half-plane, cf.~Notation 5.2.3 of \emph{op.cit.}. (Here ``$\ah, i$'' signifies that the eigenform is anti-holomorphic at the $i$-th infinite place, and holomorphic at the other one.)

 \subsubsection*{Trivial-character case} First let us suppose $\omega_{\Pi}|_{\QQ} = 1$.

 \begin{definition}
  We define
  \[
   \mathcal{I}(\Pi) \coloneqq \int_{\Gamma_1(N) \backslash \cH} \iota^*\left(\cF^{\mathrm{ah}, 1}\right)(x + iy) y^{k-2}\, \mathrm{d}x\, \mathrm{d}y,
  \]
  where $N$ is the positive integer generating the ideal $\mathfrak{N} \cap \ZZ$.
 \end{definition}

 \begin{proposition}
  The period integral $\mathcal{I}(\Pi)$ is non-zero if, and only if, $\Pi$ is distinguished (for the trivial character $\chi$).
 \end{proposition}

 \begin{proof}
  As in \cite[Theorem 5.3.2]{leiloefflerzerbes18}, we have the Asai period-integral formula
  \begin{multline*}
   \int_{\Gamma_1(N) \backslash \cH} \iota^*\left(\cF^{\mathrm{ah}, 1}\right)(x + iy)\, E^{(0)}_{1/N}(x + iy, s)\, y^{k-2}\, \mathrm{d}x\, \mathrm{d}y \\
   = \frac{N^{2s} D^{(s + k - 1)/2} \Gamma(s + k - 1) \Gamma(s)}{2^{2(s + k - 1)} \pi^{2s + k - 1}}\cdot L_{\As}^{\imp}(\Pi, s),
  \end{multline*}
  where $E^{(0)}_{1/N}(\tau, s)$ is a real-analytic Eisenstein series, and $D$ is the discriminant of $F$. We have
  \[ \Res_{s = 1} E^{(0)}_{1/N}(\tau, s) = 1\]
  (independently of $N$ and $\tau$), by Kronecker's first limit formula; so the residue of the left-hand side at $s = k$ is exactly $\mathcal{I}(\Pi)$. Thus we have
  \[ \mathcal{I}(\Pi) = \frac{N^2 D^{k/2} (k-1)!}{2^{2k} \pi^{k+1}} \cdot \Res_{s = 1} L^{\mathrm{imp}}_{\As}(\Pi, s),\]
  and in particular $\mathcal{I}(\Pi)$ is non-zero if and only if $L^{\mathrm{imp}}_{\As}(\Pi, s)$ has a pole, which we have seen occurs if and only if $\Pi$ is distinguished.
 \end{proof}

 \begin{remark}
  Note that $\mathcal{I}(\Pi)$ is a particular value of the automorphic period functional considered in (2) of \cref{thm:whenpoles}; so if $\Pi$ is not distinguished, we clearly have $\mathcal{I}(\Pi) = 0$. However, the converse implication is less obvious. In automorphic terms, we have shown that the forms $\cF^{\mathrm{ah}, i}$ are ``test vectors'' for the automorphic period.
 \end{remark}

 \subsubsection*{General $\chi$}

  We now extend the above to twisted Asai $L$-series, using twisting operators analogous to \cite[\S 7]{loefflerwilliams18} in the Bianchi case. We no longer require that $\omega_{\Pi}|_{\QQ}$ be trivial, and we consider a Dirichlet character $\chi$ of conductor $M$ with $\chi^2 \omega_{\Pi}|_{\QQ} = 1$. Let $a \in \cO_F$ be such that $a - a^{\sigma} = \sqrt{D}$ (so that $a$ generates $\cO_F / \ZZ$). Then we can consider the function
  \[
   R_{a, \chi^{-1}}(\cF^{\ah, i}) \coloneqq \sum_{u \in (\ZZ/M \ZZ)^\times} \chi(u)^{-1} \cdot \cF^{\ah, i}(\tau + \tfrac{ua}{M}),
  \]
  which is invariant under the subgroup $\Gamma_{F, 1}(M^2 \fN)$ of $\SL_2(\cO_F)$, and transforms trivially under the diamond operators $\langle d \rangle$ for $d \in (\ZZ / M^2 N \ZZ)^\times$. We then define
  \[
   \mathcal{I}_{\chi}(\Pi) = \int_{\Gamma_1(M^2 N) \backslash \cH} \iota^*\left(R_{a, \chi^{-1}} \cF^{\ah, 1}\right)(x + iy) y^{k - 2}\ \mathrm{d}x\, \mathrm{d}y.
  \]

  Exactly as above, we can identify $\mathcal{I}_{\chi}(\Pi)$ as the residue at $s = 1$ of the more general integral
  \[
   \int_{\Gamma_1(M^2 N) \backslash \cH} \iota^*\left(R_{a, \chi^{-1}} \cF^{\ah, 1}\right)(x + iy)\, E^{(0)}_{1/M^2N}(x + iy, s) y^{k - 2}\ \mathrm{d}x\, \mathrm{d}y, \
  \]
  and this integral can be evaluated as a product of exponentials and $\Gamma$ functions (which are holomorphic and nonzero at $s = 1$) times $L^{\mathrm{imp}}_{\As}(\Pi, \chi, s)$, exactly as in the Bianchi case treated in \cite{loefflerwilliams18}\footnote{A similar integral representation of the twisted Bianchi Asai $L$-function was also independently obtained by Balasubramanyam et al.~in \cite{BGV20}; for further generalizations to CM fields see \cite{namikawa22}.}. So $\mathcal{I}_{\chi}(\Pi)$ is non-zero if and only if $\Pi$ is $\chi$-distinguished.


\section{An improved $p$-adic $L$-function for quadratic Hilbert modular forms}\label{sec:improved}

 We fix a prime $p$, and an embedding $\iota : \QQbar \into \QQbar_p$. We assume $p$ is split in $F$, and we number the infinite places $\sigma_1, \sigma_2$ of $F$ and the primes $\pp_1, \pp_2$ above $p$ compatibly with $\iota$, so $\pp_i$ is the prime induced by the embedding $\iota \circ \sigma_1 : F \into \QQbar_p$.

 For any $p$-adically complete topological ring $R$, $N \ge 1$ coprime to $p$, and $\kappa : \Zp^\times \to R^\times$ a continuous character, we write $\cM_\kappa(N, R)$ for the space of $R$-valued $p$-adic modular forms of weight $\kappa$.

 \subsection{Families of Eisenstein series}

  Let $\Lambda = \Zp[[\Zp^\times]]$ be the Iwasawa algebra of $\Zp^\times$, and $\Lambda' = \Zp[[\Zp^\times \times \Zp^\times]]$. We write $\kappa$ for the canonical character $\Zp^\times \to \Lambda^\times$, and $\kappa_1, \kappa_2$ the two canonical characters into $(\Lambda')^\times$. Moreover, let $N \ge 1$ be coprime to $p$.

  \begin{definition}
   For $N$ coprime to $p$, write
   \[ \cE_{1/N}\left(\kappa_1,\kappa_2\right) \in \cM_{\kappa_1 + \kappa_2 + 1}\left(N, \Lambda'\right) \]
   for Katz' 2-parameter family of Eisenstein series of level $N$, with $q$-expansion
   \[ \sum_{\substack{u, v \in \ZZ \\ uv > 0 \\ p \nmid uv}} u^{\kappa_1} v^{\kappa_2} \sgn(u) \exp(2\pi i v / N) q^{uv}. \]
  \end{definition}

  The specialisation of this form at $(\kappa_1, \kappa_2) = (a, b)$, for any integers $a, b \ge 0$, is a classical nearly-holomorphic modular form of weight $a + b + 1$ and level $N p^2$; and this form is $p$-depleted (lies in the kernel of the operator $U_p$). In particular, the specialisation at $(0, k-1)$ is the $p$-depletion of Kato's Eisenstein series $E^{(k)}_{0, 1/N}$, and the specialisation at $(k - 1, 0)$ is the $p$-depletion of Kato's $F^{(k)}_{0, 1/N}$ (see \cite[\S 3.8]{kato04} for these definitions).

  \begin{remark}
   The projection of this Eisenstein family to a character eigenspace for the diamond operators mod $N$ is the Eisenstein family described in \S 7.3.2 of \cite{grossiloefflerzerbesLfunct}, which is a special case of the construction of  \cite[Theorem 7.6]{LPSZ1} for a particular choice of prime-to-$p$ Schwartz function $\Phi^{(p)}$. Essentially the same Eisenstein family appears in \cite{leiloefflerzerbes14} but in that paper the roles of $\kappa_1$ and $\kappa_2$ are interchanged relative to the current setting; we apologise for any confusion caused by this (inadvertent) switch.
  \end{remark}

  We are interested in an ``improved'' Eisenstein family, which lives over the smaller Iwasawa algebra $\Lambda$:

  \begin{proposition}
   There exists a $p$-adic family of Eisenstein series
   \[ \cE^{\flat}_{1/N}(\kappa) \in \cM_{\kappa}(N, \Lambda) \otimes_{\Lambda} \, \mathcal{Q}(\Lambda),
   \]
   where $\mathcal{Q}(\Lambda)$ is the total ring of fractions of $\Lambda$, whose $q$-expansion is given by
   \[ \zeta_p(1 - \kappa) + \sum_{\substack{u, v \in \ZZ \\ uv > 0,\ p \nmid u}} u^{\kappa - 1} \sgn(u) \exp(2\pi i v / N) q^{uv}.\]
   Here $\zeta_p(1 - \kappa) \in\mathcal{Q}(\Lambda)$ is the Kubota--Leopoldt $p$-adic $L$-function.
  \end{proposition}

  \begin{proof}
   Well-known.
  \end{proof}

  Observe that the $p$-depletion of $\cE^{\flat}_{1/N}(\kappa)$ is $\cE_{1/N}(\kappa - 1, 0)$, a one-parameter ``slice'' through Katz's two-parameter family; but without $p$-depleting, it is impossible to extend $\cE^{\flat}_{1/N}(\kappa)$ to a two-parameter family. Moreover, the specialisation of $\cE^{\flat}_{1/N}(\kappa)$ at an integer\footnote{We consider $\ZZ$ as a subset of the characters of $\Zp^\times$ in the obvious fashion, so ``at $k$'' here and below  means at the character $x \mapsto x^k$.}  $k \ge 1$ is the ordinary $p$-stabilisation of Kato's $F^{(k)}_{0, 1/N}$.

  \begin{notation}
   We write $\ell(-)$ for the ``logarithm'' map on characters of $\Zp^\times$, mapping a character $\sigma$ to $\sigma'(1)$; this is a rigid-analytic function on $\mathcal{W} = \left(\operatorname{Spf} \Lambda\right)^{\mathrm{rig}}$ which has a simple zero at the trivial character, and takes the value $k$ at $x \mapsto x^k$.
  \end{notation}
  
  Since $\ell$ is a local uniformizer at $\kappa = 0$ (i.e.~at the trivial character), we can use it to define the \emph{residue} of an element of $\mathcal{Q}(\Lambda)$.

  \begin{proposition}
   The pole of $\cE^{\flat}_{1/N}(\kappa)$ at $\kappa = 0$ is simple, and its residue is the constant $p$-adic modular form $1 - \tfrac{1}{p}$.
  \end{proposition}

  \begin{proof}
   This is clear from the $q$-expansion formula above, since the non-constant coefficients are in $\Lambda$, and the $p$-adic zeta function has a simple pole with residue $1 - \tfrac{1}{p}$.
  \end{proof}


 \subsection{The improved $p$-adic $L$-function}

  Let $\Pi$ be the unitary cuspidal automorphic representation of $\GL_2(\AA_F)$ generated by a holomorphic Hilbert modular newform of weight $(k, k)$, for some integer $k \ge 2$, with trivial central character. Assume that $p$ is split in $F$, and that $\Pi$ is ordinary at $p$.

   Let $\underline{\Pi}$ be a Hida family through $\Pi$, of weight $(k + 2\kappa, k)$, and let $\nu_{\underline{\Pi}}$ be a choice of basis of the associated coherent-cohomology eigenspace as in \cite[\S 7.4]{grossiloefflerzerbesLfunct}. In general this is only defined over a finite integral ring extension of $\Lambda$; but for simplicity of notation, we suppose it is defined over $\Lambda$ itself (and we similarly assume, for simplicity of notation, that the Hecke eigenvalues of $\Pi$ are in $\QQ$, although of course the results remain valid without this restriction).

   In \emph{op.cit.}, we defined an object $\iota^\star\left(\nu_{\underline{\Pi}}\right)$, which is a linear functional $\cM_{2\kappa}(N, \Lambda) \to \Lambda$; by construction, specialising this at integer values of $\kappa$ gives the linear functionals obtained by pairing with pullbacks of $U_{\pp_1}^{t}$-eigenclasses in $H^1$ of the Hilbert modular surface associated to the specialisations of $\uPi$. We recall the following construction from \emph{op.cit.}:

  \begin{definition}
   Define the $p$-adic $L$-function
   \[
    \cL_{p,\As}(\uPi)(\kappa, \sigma) = (\star) \cdot
    \Big\langle \iota^\star\left(\nu_{\underline{\Pi}}\right),\,
    \cE_{1/N}(\kappa - \sigma, \sigma + \kappa - 1)
    \Big\rangle \in \Lambda'.
   \]
   where $(\star) = \tfrac{p + 1}{p} \cdot (\sqrt{D})^{1 - (k + \kappa + \sigma)} \cdot (-1)^\sigma$.
  \end{definition}

  \begin{remark}
   By \cite[Theorem C]{grossiloefflerzerbesLfunct}, the values of this analytic function at $(\kappa, \sigma) = (a, s) \in \ZZ^2$, with $(a, s)$ satisfying $1 - a \le s \le a$, interpolate the values $L_{\As}(\Pi[a], s)$, where $\Pi[a]$ is the specialisation of $\uPi$ in weight $(k + 2a, k)$.
  \end{remark}

  We now define the improved $p$-adic $L$-function by replacing the Eisenstein measure by the ordinary Eisenstein family:

  \begin{definition}
   Define the improved $p$-adic $L$-function $\cL^\flat_{p,\As}(\uPi) \in \mathcal{Q}(\Lambda)$ by
   \[
    \cL^\flat_{p,\As}(\uPi)(\kappa) =
     \tfrac{p + 1}{p} \cdot (\sqrt{D})^{-k} \cdot \Big\langle
     \iota^\star\left(\nu_{\underline{\Pi}}\right),\,
     \cE^{\flat}_{1/N}(2\kappa)
    \Big\rangle.
   \]
  \end{definition}

 \begin{note}
  Observe that a priori, we have $\cL^\flat_{p,\As} \in \mathcal{Q}(\Lambda)$, with a possible simple pole at $\kappa = 0$ inherited from $\cE^{\flat}_{1/N}$.
 \end{note}

 \begin{theorem}
  \label{thm:nopole}
  The following are equivalent:
  \begin{itemize}
   \item $\cL^\flat_{p,\As}(\uPi)$ has a pole at $\kappa = 0$.
   \item $\Pi$ is distinguished.
  \end{itemize}
 \end{theorem}
 
 \begin{proof}
  Since  $\cL^\flat_{p,\As}(\kappa)$ can have at worst a simple pole at $\kappa=0$, the statement will follow if we can show that
  \[ \left(\ell(\kappa) \cdot \cL^\flat_{p,\As}(\kappa)\right)|_{\kappa=0}=0\]
  unless $\kappa$ is distinguished. However, this limit is by construction equal to
  \[
   \tfrac{p + 1}{p} \cdot (\sqrt{D})^{-k} \left\langle \iota^\star\left(\nu_{\underline{\Pi}}\middle|_{\kappa = 0}\right),\,
   \lim_{\kappa \to 0} \ell(\kappa) \cdot \cE^{\flat}_{1/N}(2\kappa) \right\rangle, \]
  where $\ell(-)$ is the ``logarithm'' function on weight space (which maps a character $\kappa$ to $\kappa'(1)$, so it takes the value $k$ at the character $x \mapsto x^k$; in particular it has a simple zero at the trivial character).

  we are abusing notation slightly by treating $\kappa$ as a 
  The vector $\nu_\Pi \coloneqq \nu_{\underline{\Pi}}|_{\kappa = 0}$ is a $\Qp$-basis of the 1-dimensional $\Pi$-eigenspace in the coherent $H^1$ of the Hilbert modular surface (with a suitable coefficient system depending on $k$). This is canonically the base extension to $\Qp$ of a one-dimensional $\QQ$-vector space; and the base-extension to $\CC$ of the same space is spanned by $\cF^{\ah, 1}$. Hence we can find scalars $\Omega_p \in \Qp^\times$ and $\Omega_\infty \in \CC^\times$ such that
  \[
   \frac{\nu_\Pi}{\Omega_p} = \frac{\cF^{\mathrm{ah}, 1}}{\Omega_\infty} \in H^1(X_1(\mathfrak{N}) / \QQ, \omega^{(2-k, k)})[\Pi].
  \]
  On the other hand, the limit on the right-hand side of the pairing is the non-zero constant $\tfrac{1}{2}\left(1 - \tfrac{1}{p}\right)$, considered as a weight 0 $p$-adic modular form. Thus we have shown
  \[ \frac{\left(\ell(\kappa) \cdot \cL^\flat_{p,\As}\right)|_{\kappa=0}}{\Omega_p(\Pi)} = \tfrac{1}{2} (\sqrt{D})^{-k} \left(1 - \tfrac{1}{p^2}\right)\cdot \frac{\mathcal{I}(\Pi)}{\Omega_\infty(\Pi)}. \]
  Since $\mathcal{I}(\Pi)$ is non-zero if and only if $\Pi$ is distinguished, the proof is complete.
 \end{proof}

 \subsection{Interpolating property of the improved $p$-adic $L$-function}

  \begin{notation}
  	For $i=1,2$, write $\alpha_i(m)$ and $\beta_i(m)$ for the $U_{\pp_i}$-eigenvalues of $\Pi[m]$, with $\alpha_1(m)$ being the distinguished unit root.
  \end{notation}

  Our normalisations here are such that $\alpha_1(m)$ and $\beta_1(m)$ both have complex absolute value $p^{(k + 2m - 1)/2}$, whereas $\alpha_2(m)$ and $\beta_2(m)$ have complex absolute value $p^{(k - 1)/2}$ (independent of $m$). Moreover, all the parameters are $p$-adically integral (for our distinguished embedding into $\Qp$)\footnote{The parameters in general lie in a number field, which can have large degree; and we are not claiming that the parameters are integral at all primes of this field above $p$ -- only that they are integral at a specific prime corresponding to our embedding into $\Qp$.}; and $\alpha_1(m)$, $\alpha_2(m)$ and $\beta_2(m)$ are all analytically varying in $m$, while $\beta_1$ is not.

  We recall the formula (Theorems C and D of \cite{grossiloefflerzerbesLfunct}) relating the two-variable $p$-adic $L$-function $\cL_{p, \As}(\uPi)$ to the critical values of the complex $L$-function. Specialised to the present situation, this reads as follows: for all $m \in \ZZ_{> 0}$, we have
  \begin{equation}
   \label{eq:interp}
   \frac{\cL_{p, \As}(\uPi)(m, 1 - m)}{\Omega_p(\Pi[m])} =
   \cE_p(\As(\Pi[m]), 1 - m) \cdot L_{\As}^{\imp}(\Pi[m], 1 - m)^{\alg},
  \end{equation}
  where
  \[ L_{\As}^{\imp}(\Pi[m], 1 - m)^{\alg} \coloneqq \frac{(k-1)!}{2^{(k + 2m - 2)} i^{(1-k-2m)} (-2\pi i)^{k+1}} \cdot \frac{L_{\As}^{\imp}(\Pi[m], 1 - m)}{\Omega_\infty(\Pi[m])} \in \overline{\QQ}.\]
  Here $\Omega_p(\Pi[m])$ and $\Omega_\infty(\Pi[m])$ are $p$-adic and complex periods for $\Pi[m]$ compatible with the weight $m$ specialisation of $\eta_{\uPi}$, as in the previous section; and $\cE_p(\As(\Pi[m]), 1 - m)$ is an Euler factor, which can be written explicitly as
  \[ \cE_p(\As(\Pi[m]), 1 - m) =
  \left(1 - \frac{\alpha_2(m)}{\alpha_1(m)}\right)
  \left(1 - \frac{\beta_2(m)}{\alpha_1(m)}\right)
  \left(1 - \frac{p^{2m-1} \alpha_2(m)}{\alpha_1(m)}\right)
  \left(1 - \frac{p^{2m-1} \beta_2(m)}{\alpha_1(m)}\right).\]
  Note that the first two bracketed terms vary analytically with $m$, while the other two clearly do not.

  \begin{lemma}
   For all $m\in\ZZ_{>0}$, we have the formulae
   \[ \frac{\cL^\flat_{p,\As}(\uPi)(m)}{\Omega_p(\Pi[m])}
     = \left(1 - \frac{p^{2m-1} \alpha_2(m)}{\alpha_1(m)}\right)
       \left(1 - \frac{p^{2m-1} \beta_2(m)}{\alpha_1(m)}\right)\cdot L_{\As}^{\imp}(\Pi[m], 1 - m)^{\alg} \]
   and
   \[ \cL_{p,\As}(\uPi)(m, 1 - m) = \left(1-\frac{\alpha_2(m)}{\alpha_1(m)}\right)\left(1-\frac{
   \beta_2(m)}{\alpha_1(m)}\right)\cdot \cL^\flat_{p,\As}(\uPi)(m). \]
  \end{lemma}

  \begin{proof}
   Note that the first formula follows from the second formula and \eqref{eq:interp}, since the assumption $m \in \ZZ_{>0}$ implies that the factor $ \left(1-\frac{\alpha_2(m)}{\alpha_1(m)}\right)\left(1-\frac{
   \beta_2(m)}{\alpha_1(m)}\right)$ is invertible ($\alpha_2$ and $\beta_2$ have different complex absolute values from $\alpha_1$), so assuming the second formula we can write
   \begin{align*}
    \frac{\cL^\flat_{p,\As}(\uPi)(m)}{\Omega_p(\Pi[m])} &= 
    \frac{\cL_{p,\As}(\uPi)(m, 1 - m)}{\left(1-\frac{\alpha_2(m)}{\alpha_1(m)}\right)\left(1-\frac{
           \beta_2(m)}{\alpha_1(m)}\right)} \\
    &=
    \frac{\cE_p(\As(\Pi[m]), 1 - m)}{\left(1-\frac{\alpha_2(m)}{\alpha_1(m)}\right)\left(1-\frac{
       \beta_2(m)}{\alpha_1(m)}\right)} \cdot L_{\As}^{\imp}(\Pi[m], 1 - m)^{\alg} \\
    &= \left(1 - \tfrac{p^{2m-1} \alpha_2(m)}{\alpha_1(m)}\right)
           \left(1 - \tfrac{p^{2m-1} \beta_2(m)}{\alpha_1(m)}\right)\cdot L_{\As}^{\imp}(\Pi[m], 1 - m)^{\alg}.
   \end{align*}
   So it suffices to prove the second formula. Recall that
   \[  
    \cL^\flat_{p,\As}(\uPi)(m) = (\star) \cdot \Big\langle
    \iota^\star\left(\nu_{\underline{\Pi}}\right),\,
    \cE^{\flat}_{1/N}(2m) \Big\rangle
   \]
   whereas
   \[  
    \cL_{p,\As}(\uPi)(\kappa) = (\star) \cdot \Big\langle
    \iota^\star\left(\nu_{\underline{\Pi}}\right),\, \cE_{1/N}(2m-1, 0)
    \Big\rangle,
   \]
   where $(\star)$ in both formulae denotes the same constant $\tfrac{p + 1}{p} \cdot (\sqrt{D})^{-k}$, and $\cE_{1/N}(2m)$ is an abbreviation for $\cE_{1/N}(0, 2m-1)$. These formulae differ only in the choice of Eisenstein series: the Eisenstein series $\cE_{1/N}(2m-1,0)$ in the latter formula is the $p$-depletion of the ordinary Eisenstein $\cE^\flat_{1/N}(2m)$.
   
   We can compute both of these cup-products using global zeta-integrals, as in \cite{grossiloefflerzerbesLfunct}, which factor as products of local zeta-integrals. The local integrals are the same at all places other than $p$, so it suffices to compare the local integrals at $p$. We explain below, in the Appendix to this paper, how to compute these zeta-integrals using a symmetry property of the $\GL_2 \times \GL_2$ zeta-integral (analogous to a related computation for $\GSp_4$ zeta integrals carried out in \cite[\S 8]{LPSZ1}); comparing these formulae gives the factor $\left(1-\frac{\alpha_2(m)}{\alpha_1(m)}\right)\left(1-\frac{
              \beta_2(m)}{\alpha_1(m)}\right)$.
  \end{proof}
  
  Since $\frac{\alpha_2(2k)}{\alpha_1(2k)}$ varies $p$-adically in $k$, and the positive integers are dense in $\Spec \Lambda$, we obtain the following factorisation:

  \begin{proposition}
   \label{prop:factorisation}
   We have an identity in $\Lambda$
   \[
    \cL_{p,\As}(\uPi)(\kappa, 1 - \kappa)=\left(1-\frac{\alpha_2(\kappa)}{\alpha_1(\kappa)}\right) \left(1-\frac{\beta_2(\kappa)}{\alpha_1(\kappa)}\right)\cdot \cL^\flat_{p,\As}(\uPi)(\kappa).
   \]
  \end{proposition}

  If $\Pi$ is a base-change of a form of \emph{trivial} character, then it is not distinguished, so $\cL^\flat_{p,\As}(\kappa)$ is regular at $\kappa = 0$; however, $\alpha_1(0)$ is equal to one of $\alpha_2(0)$ and $\beta_2(0)$, so one of the bracketed terms vanishes. Hence we obtain the following consequence:

  \begin{corollary}
   If $\Pi$ is a non-distinguished base-change form, then $\cL_{p,\As}(\uPi)$ vanishes at $(0, 1)$.
  \end{corollary}

  Note that we cannot immediately conclude the converse statement that $\cL_{p,\As}(\uPi)(0,1) \ne 0$ when $\Pi$ \emph{is} distinguished, because the order of vanishing of $\left(1-\frac{\alpha_2(\kappa)}{\alpha_1(\kappa)}\right) \left(1-\frac{\beta_2(\kappa)}{\alpha_1(\kappa)}\right)$ at $\kappa = 0$ might be higher than 1. (This seems unlikely, but we do not know how to rule it out at present.)

 \subsection{A second improved $p$-adic $L$-function}

  One can also make an analogous construction with a second family of Eisenstein series $\cE^{\sharp}_{1/N}(\kappa)$ whose $p$-depletion is $\cE_{1/N}(0, \kappa-1)$, a different one-parameter slice of Katz's two-parameter family. This gives a $p$-adic $L$-function $\cL^\sharp_{p,\As}(\uPi)$, which is related (by the same Euler factor as before) to the values of $\cL_{p,\As}(\uPi)(\kappa, \sigma)$ along the line $\sigma = \kappa$, rather than $\sigma = 1-\kappa$.

  However, except in the case $\fN = 1$ (when both constructions trivially coincide), the ``sharp'' $p$-adic $L$-function does not have the same relation to distinction of representations as the ``flat'' version. This is because the constant term of $\cE^{\sharp}_{1/N}(\kappa)$ is a $p$-adic measure interpolating the values of the function $\zeta^*(\tfrac{1}{N}, s)$ (in the notation of \cite[\S 3]{kato04}), defined by the Dirichlet series $\sum_{n \ge 1} \exp\left(\tfrac{2\pi i n}{N}\right) n^{-s}$; and for $N > 1$ this function, and its $p$-adic analogue, are regular at $s = 1$.

  From the viewpoint of $L$-functions, this corresponds to the fact that the Euler factors $C_\ell(\Pi, s)$ relating the primitive and imprimitive Asai $L$-functions can vanish at $s = 0$ (unlike at $s = 1$). Since the primitive Asai $L$-function satisfies a functional equation, \cref{thm:whenpoles} shows that $\Pi$ is distinguished if and only if $L_{\As}(\Pi, 0) \ne 0$; but the possible vanishing of these ``bad'' Euler factors at places dividing $N$ means that one cannot straightforwardly identify distinguished representations purely from the behaviour of their \emph{imprimitive} $L$-functions at $s = 0$.

 \section{Twisted analogues}

  We briefly sketch how the results of the previous section can be extended to twisted Asai $L$-functions. Given a character $\chi$ of conductor $M$ coprime to $p$, we replace the pullback $\iota^\star\left(\nu_{\underline{\Pi}}\right)$ with pullback along the more general correspondences $\iota_{M, N, a}$ described in \cite{leiloefflerzerbes18}. Considering a weighted sum of such correspondences, with $\chi$ as the weighting factor, and pairing with the Eisenstein family at level $M^2 N$ gives a $p$-adic $L$-function interpolating the $\chi$-twisted Asai $L$-series, exactly as in the Bianchi setting studied in \cite{loefflerwilliams18}.

  Arguing exactly as above, we conclude the following:

  \begin{corollary}
   If $\chi$ is a character satisfying $\omega_{\Pi}|_{\QQ} = \chi^{-2}$, but $\Pi$ is not $\chi$-distinguished, then $\cL_{p,\As}(\uPi, \chi)$ vanishes at $(\kappa, \sigma) = (0, 1)$.\qed
  \end{corollary}

\appendix

\section{A zeta-integral computation}

 \subsection{Setup}

  In this section $K$ denotes a local field of characteristic 0, and $q, \pp, \varpi, |\cdot|$ etc have their usual meanings. We let $\psi_1, \psi_2$ be additive characters of $K$ with $\psi_2 = \psi_1^{-1}$. (This is natural in our global application, where $\psi_i$ are the restrictions to the two places above $p$ of an additive character of $\AA_F / F$ trivial on $\AA_\QQ$.) 
  
  We let $H = \GL_2(K)$ and we let $\pi_1$ and $\pi_2$ be two infinite-dimensional (equivalently, generic) irreducible representations of $H$, with central characters $\omega_1$ and $\omega_2$. We shall suppose $\pi_1$ is a principal-series representation $I(\mu_1, \nu_1)$, where $\mu_1, \nu_1$ are smooth characters of $K^\times$ (and $I(-)$ denotes normalised induction from the Borel subgroup $B_H$). We let $\omega = \omega_1 \omega_2$.
  
  For concreteness, we will normalise all our Haar measures on $\GL_2(K)$ and its subgroups so that the $\cO$-points have volume 1. We also assume $K$ is unramified and the characters $\psi_i$ have conductor 1, so our Haar measures are self-dual with respect to $\psi_i$. 
   
 \subsection{A symmetry property of the Rankin--Selberg zeta integral}
  
  \begin{definition}[Godement--Siegel sections]
   For $\Phi \in \cS(K^2)$ (Schwartz functions on $K^2$), we define a function on $H$ by
   \[ 
    f^{\Phi}(h; \omega, s) = |\det h|^s \int_{K^\times} \Phi((0, t) h) |t|^{2s}\omega(t)\, \mathrm{d}^\times h.
   \]
   This transforms in $h$ as an element of the induced representation $\sigma_s \coloneqq I(|\cdot|^{s - 1/2}, |\cdot|^{1/2 - s} \omega^{-1})$.
  \end{definition}
  
  Our goal is to study values of the $\GL_2 \times \GL_2$ zeta integral
  \[ Z(W_1, W_2, \Phi; s) \coloneqq \int_{Z_HN_H \backslash H} W_1(h) W_2(h) f^{\Phi}(h; \omega, s)\, \mathrm{d}h, \]
  where $W_i$ are functions in the Whittaker models $\cW(\pi_i, \psi_i)$. Note that this integral uses the Whittaker model of $\pi_1$ and the principal-series model of $\sigma_s$; and we shall see below that it is proportional to a second integral in which the models are switched around, so it is the principal-series model of $\pi_1$ and the Whittaker model of $\sigma_s$ that appear.
  
  \begin{definition}[Intertwining maps] \
   \begin{enumerate}[(i)]
    \item For $f_1 \in I(\mu_1, \nu_1)$, we write $W_{f_1}(h) = \int_K f_1(\stbt{}{1}{-1}{} \stbt{1}{t}{}{1} h) \psi_2(t) \, \mathrm{d}t$, which is an element of $\cW(\pi_1, \psi_1)$. 
    
    \item Similarly, for $\Phi \in \cS(K^2)$, we write $W^{\Phi}(h; \omega, s) = \int_K f^{\Phi}(\stbt{}{1}{-1}{} \stbt{1}{t}{}{1} h; \omega, s) \psi_2(t) \, \mathrm{d}n$, which lies in the Whittaker model $\cW(\sigma_s, \psi_1)$.
   \end{enumerate}
  \end{definition}
  
  Note that $W^{\Phi}(h; \omega, s)$ is always holomorphic in $s$, although $f^{\Phi}(h; \omega, s)$ may not be (cf.~\cite[\S 8.1]{LPSZ1}). The integral defining $W_{f_1}(h)$ is only absolutely convergent if $|\tfrac{\mu_1}{\nu_1}(\varpi)| < 1$; however, it can be extended to all principal-series representations by intepreting $f_1$ as a flat section of the family of representations $I(\mu_1 |\cdot|^{\xi}, \nu_1 |\cdot|^{-\xi})$, for an auxiliary $\xi \in \CC$,  and continuing analytically in $\xi$ (cf.~\cite[ch.~4, Eq.~(5.32)]{bump97}). The integrals in the next proof should also be understood in the same regularised sense; one can check that all the integrals are absolutely convergent for suitable values of $s$ and $\xi$.

  \begin{proposition}
   \label{prop:zeta1}
   For any $f_1 \in I(\mu_1, \nu_1)$, $W_2 \in \cW(\pi_2, \psi_2)$, and $\Phi \in \cS(F^2)$, we have
   \[ Z(W_{f_1}, W_2, \Phi; s) = \frac{\omega(-1)}{\gamma(\nu_1 \times \pi_2, \psi_2, s)} \int_{Z_H N_H \backslash H} f_1(h) W_2(h) W^{\Phi}(h; \omega, s)\, \mathrm{d}h.\]
  \end{proposition}

  \begin{proof}
   Substituting the definition of $W_{f_1}$ in terms of $f_1$, we have
   \[ Z(W_f, W_2, \Phi; s) = \int_{Z_H\backslash H} f_1(J h) W_2(h) f^{\Phi}(h; \omega, s) \, \mathrm{d}h,\quad J = \stbt{}{1}{-1}{}. \]
   Since $f_1(J h) f^{\Phi}(h; \omega, s)$ transforms via $|\cdot|^{s - \tfrac{1}{2}} \nu_1$ under left-translation by $\stbt{\star}{}{}{1}$, we can write this as
   \[ \int_{T_H \backslash H} \left(\int_{K^\times} W_2(\stbt{t}{}{}{1}h) \nu_1(t)|t|^{s-1/2}\, \mathrm{d}^\times t\right) f_1(J h) f^{\Phi}(h; \omega, s)\, \mathrm{d}h.\]
   where $T_H$ is the diagonal torus. The inner integral is equal to \[ \frac{1}{\gamma(\nu_1 \times \pi_2, \psi_2, s)} \int W_2(\stbt{t}{}{}{1} J h)|t|^{1/2 - s} \omega_2^{-1}\nu_1^{-1}(t)\, \mathrm{d}^\times t,\]
   by the functional equation for the $\GL_1$ zeta integral. So we can rewrite the integral as
   \[ \frac{\omega(-1)}{\gamma(\nu_1 \times \pi_2, \psi_2, s)} \int_{T_H\backslash H} \left(\int_{K^\times} W_2(\stbt{t}{}{}{1}h) \nu_1(t))|t|^{1/2 - s} \omega_2^{-1}\nu_1^{-1}(t)\, \mathrm{d}^\times t\right) f_1(h) f^{\Phi}(J h; \omega, s)\, \mathrm{d}h,\]
   and the same argument in reverse now writes this as
   \[ \frac{\omega(-1)}{\gamma(\nu_1 \times \pi_2, \psi_2, s)} \int_{Z_H N_H \backslash H} W_2(h) f_1(h) W^{\Phi}(h; \omega, s)\, \mathrm{d}h\]
   as claimed.
  \end{proof}

  \begin{remark}
   Note that both sides of \cref{prop:zeta1} define $H$-invariant linear functionals on $\pi_1 \otimes \pi_2 \otimes \sigma_s$. For almost all $s$ the space of such $H$-invariant functionals is one-dimensional, by the results of \cite[\S 14]{jacquet72}; so it is clear that both sides must be related by a meromorphic function of $s$, independent of $(W_1, W_2, \Phi)$. In particular, if the $\pi_i$ are both unramified (which is the only case we actually need for our applications), we can simply evaluate both integrals concerned on the spherical vectors to obtain a rather simpler proof of the proposition in this case.
  \end{remark}

 \subsection{A particular case}

  We now suppose $\nu_1$ is unramified (but $\mu_1$ and $\pi_2$ may be ramified); and let $c = \max(1, \operatorname{cond}(\mu_1))$. It follows that for all $r \ge c$, the invariants of $\pi$ under the group $K_1(\pp^r) \coloneqq \{ h \in \GL_2(\cO_K) : h = \stbt{*}{*}{0}{1} \bmod \pp^r\}$ are non-zero, and this space contains a unique eigenvector for the double-coset operator $U = [K_1(\pp^r) \stbt{\varpi}{}{}{1} K_1(\pp^r)]$ with eigenvalue $q^{1/2} \nu_1(\varpi)$.

  \begin{notation}
   Let $W_{\nu_1}$ be the unique basis of the 1-dimensional space $\cW(\pi_1, \psi_1)^{K_1(\pp^c)}[U = q^{1/2} \nu_1(\varpi)]$ such that $W_{\nu_1}(1) = 1$. 
   
   For $r \ge c$, we let $W'_{\nu_1}[r] = (\tfrac{q^{1/2}}{\nu_1(\varpi)})^r \stbt{0}{-1}{\varpi^r}{0} W_{\nu_1}$, which is an eigenvector for the dual Hecke operator $U'$ with the same eigenvalue $q^{1/2} \nu_1(\varpi)$. (The normalisation $(\tfrac{q^{1/2}}{\nu_1(\varpi)})^r$ makes these compatible under the trace maps.)
  \end{notation}

  \begin{lemma}
   Let $f'_{\nu_1}[r]$ be the unique function in $I(\mu_1, \nu_1)$ which vanishes outside $B_H \cdot K_1(\pp^r)$ and takes the value $q^r \mu_1(-1)$ at the identity. Then $W_{f'_{\nu_1}[r]} = W'_{\nu_1}[r]$.
  \end{lemma}
  
  \begin{proof}
   We have a Godement--Siegel section map from Schwartz functions to $I(\mu_1, \nu_1)$, as usual. Consider the auxiliary function $\Psi_r = \mu_1(-1) \cdot q^{r - 1}(q - 1) \cdot \ch( (1 + \pp^r) \times \cO)$. Then we compute that $W^{\Psi_r}(\stbt{\varpi^k}{}{}{1}) = q^{-k/2} \nu_1(\varpi^k)$, using Remark 8.2 of \cite{LPSZ1}, so $\Psi_r$ maps to the normalised eigenfunction $W_{\nu_1}$.

   Thus $\stbt{0}{-1}{\varpi^r}{0} f^{\Psi_r}$ is a scalar multiple of the characteristic function stated, and it remains to compute its value at the identity. From the definitions we compute that this value is $(q^{1/2} \nu_1(\varpi))^r \mu_1(-1)$; thus scaling by $(\tfrac{q^{1/2}}{\nu_1(\varpi)})^r$ gives the result claimed.
  \end{proof}

  \begin{proposition}
   \label{prop:zeta2}
   For any $W_2 \in \cW(\pi_2)$ and $\Phi \in \cS(K^2)$, the sequence
   \[ \left(Z(W'_{\nu_1}[r], W_2, \Phi; s)\right)_{r \ge c} \]
   stabilises for $r \gg 0$, and its limiting value is given by
   \[ \frac{q}{q+1} \cdot \frac{1}{\gamma(\nu_1 \times \pi_2, \psi_1, s)}\cdot \int_{K^\times}  W_2(\stbt{y}{}{}{1})W^{\Phi}(\stbt{y}{}{}{1}; \omega, s) \frac{\mu_1(y)}{|y|^{1/2}}\, \mathrm{d}^\times y. \]
  \end{proposition}

  \begin{proof}
   We use \cref{prop:zeta1} to express $Z(W'_{\nu_1}[r], W_2, \Phi; s)$ as an integral in terms of $f'_{\nu_1}[r]$. This we can compute as follows:
   \begin{multline*}
    \int_{Z_H N_H \backslash H} f_1(h) W_2(h) W^{\Phi}(h; \omega, s)\, \mathrm{d}h \\
    = \int_{B_H \backslash H}\left(
     \int_{K^\times} f'_{\nu_1}[r](\stbt{y}{}{}{1} h)W_2(\stbt{y}{}{}{1} h) W^{\Phi}(\stbt{y}{0}{0}{1}h; \omega, s) \mathrm{d}^\times y\right)\mathrm{d}h \\
     \int_{B_H \backslash H}\left(
    \int_{K^\times} \frac{\mu_1(y)}{|y|^{1/2}} W_2(\stbt{y}{0}{0}{1} h) W^{\Phi}(\stbt{y}{0}{0}{1}h; \omega, s) \mathrm{d}^\times y\right) f'_{\nu_1}[r](h) \,\mathrm{d}h.
   \end{multline*}
   Since $f'_{\nu_1}[r]$ is supported in arbitrarily small neighbourhoods of the identity in $B_H \backslash H$ for $r \gg 0$, the outer integral over $B_H \backslash H$ is just the inner integral for $h = 1$ multiplied by the factor $f'_{\nu_1}[r](1) \cdot \vol\left(B_H K_1(\pp^r)\right)$, for the unramified Haar measure on $B_H \backslash H$. This volume is $1 / (q^{r-1} (q + 1))$, and we computed $f'_{\nu_1}[r](1)$ above, which gives
   \[ f'_{\nu_1}[r](1) \cdot \vol\left(B_H K_1(\pp^r)\right) = \tfrac{q}{q+1} \cdot \mu_1(-1).\]
   Putting this together with \cref{prop:zeta1}, and noting that $\mu_1(-1) \omega(-1) = \omega_2(-1) = \frac{\gamma(\nu_1 \times \pi_2, \psi_1, s)}{\gamma(\nu_1 \times \pi_2, \psi_2, s)}$, gives the result.
  \end{proof}

  \begin{remark}
   This is an analogue for the $\GL_2 \times \GL_2$ Rankin--Selberg zeta integral of \cite[Prop 8.15]{LPSZ1} in the setting of $\GSp_4 \times \GL_2$ zeta-integrals.
  \end{remark}

 \subsection{Unramified case}
  
  In our global situation, we want to compare the values $Z(W_1, W_2, \Phi)$ for two different triples $(W_1, W_2, \Phi)$. In both cases $\pi_i$ are unramified representations, $W_2$ is the unramified Whittaker function, and $W_1$ a $U'$-eigenvector; and we want to compare the values for two different $\Phi$, one corresponding to the $p$-depleted Eisenstein series $\cE$, and the other to the ordinary Eisenstein series $\cE^\flat$. We shall do this using the formula above, labelling the characters from which $\pi_1$ is induced so that $W_1$ corresponds to the vector $W_{\nu_1}'[r]$ above.
  
  In both cases, the values of the Whittaker function $W^{\Phi}$ along the torus $\stbt{y}{}{}{1}$ correspond to the $q$-expansion coefficients of the Eisenstein series at $p$. For the $p$-depleted Eisenstein series, these vanish outside $\cO_K^\times$, so the integrand is the indicator function of $\cO_K^\times$ and the integral is simply 1; this recovers, by a rather different method, the interpolation formula for $p$-adic $L$-functions computed in \cite[\S 6.1]{chenhsieh20} that we quoted in the proof of \cite[Theorem C]{grossiloefflerzerbesLfunct}.
  
  For the $p$-ordinary Eisenstein series, $W^{\Phi}(\stbt{y}{0}{0}{1})$ is zero for $|y| > 1$, but for $y \in \cO_K$ its values agree with the unramified character $\omega^{-1} |\cdot|^{1-s}$. Hence the integral of \cref{prop:zeta2} is the unramified $\GL_2$ zeta integral computing $L(\nu_1^\vee \times \pi_2^\vee, 1-s)$. In particular, the ratio of the zeta-integrals for the two test data is precisely $L(\nu_1^\vee \times \pi_2^\vee, 1-s) ^{-1}$.

%

\providecommand{\bysame}{\leavevmode\hbox to3em{\hrulefill}\thinspace}
\renewcommand{\MR}[1]{%
 MR \href{http://www.ams.org/mathscinet-getitem?mr=#1}{#1}.
}
\providecommand{\href}[2]{#2}
\newcommand{\articlehref}[2]{\href{#1}{#2}}

\end{document}